\documentclass[12pt]{article}
\usepackage{amsmath,amssymb,amsthm}
\usepackage{amsfonts, amsmath}
\newtheorem{theorem}{Theorem}[section]
\newtheorem{lemma}{Lemma}[section]

\theoremstyle{definition}
\newtheorem{definition}{Definition}[section]
\theoremstyle{remark}

\newtheorem{example}{Example}[section]
\numberwithin{equation}{section}

\begin{document}
\label{pageinit}

\date{}

\title{Theories of truth for countable languages which conform to classical logic}

\author{Seppo Heikkil\"a $^\star$}

\maketitle

Department of Mathematical Sciences, University of Oulu

BOX 3000, FIN-90014, Oulu, Finland

\markboth{}{Theories of truth}



\begin{abstract}
\noindent Every countable language which conforms to classical logic is shown to have an extension  which conforms to classical logic, and  
has a definitional theory of truth. That extension has a semantical theory of truth, if every sentence of the
object language is valuated by its meaning either as true or as false. 
These theories contain both a truth predicate and a non-truth predicate.
Theories are equivalent when the sentences of the object language are valuated by their meanings.

\end{abstract}

\baselineskip 14pt

 \section{Introduction}\label{S1}
Based on 'Chomsky Definition' (cf. \cite{C}) we assume that a language is a nonempty countable set of sentences with finite length, and formed by a countable set of elements.
A theory of syntax is also assumed to provide a language with rules to construct well-formed sentences, formulas etc.

A language is said to conform to classical logic if it has, or if it can be extended to have at least the following properties ('iff' means 'if and only if'):
\smallskip

(i) It contains logical symbols $\neg$ (not), $\vee$ (or), $\wedge$ (and), $\rightarrow$ (if...then),  $\leftrightarrow$ (iff), $\forall$ (for all) and $\exists$ (exists), and the following sentences: If $A$ and $B$ are (denote) sentences of the language, so are $\neg A$, $A\vee B$, $A\wedge B$, $A\rightarrow B$ and $A\leftrightarrow B$.
If $P(x_1,\dots,x_m)$, $m\ge 1$, is a formula of the language with $m$ free variables, then  $P$ is called a predicate with arity $m$ and domain $D_P=D_P^1\times\cdots\times D_P^m$, where each $D_P^i$ is a subset of a set $D$ of objects, called the domain of discourse, if the following properties hold.\\
(p1) Every object of $D$ is named by a term. Denote by $N_P^i$ the set of those terms which name the objects of $D_P^i$, $i=1,\dots,m$.\\
(p2) $P(b_1,\dots,b_m)$ is a sentence of that language obtained from $P(x_1,\dots,x_m)$ by substituting for each $i=1,\dots,m$ a term $b_i$   of $N_P^i$ for $x_i$ in every free occurrence of $x_i$ in $P$.\\
 If $P$ is a predicate with arity $m\ge 1$, then the sentences  $q_1x_1\dots q_mx_m P(x_1,\dots,x_m)$, where each $q_i$ is either $\forall$ or $\exists$, are in the language. $\neg P$ is also a predicate with arity $m$ and domain $D_P$.
\smallskip

(ii)  The sentences of that language are so valuated as true or as false that the following rules of classical logic are valid: If $A$ and $B$ denote sentences of the language, then
$A$ is true iff $\neg A$ is false, and $A$ is false iff $\neg A$ is true; $A\vee B$ is true iff $A$ or $B$ is true, and false iff  $A$ and $B$ are false; $A\wedge B$ is true iff  $A$ and $B$ are true, and false iff $A$ or $B$ is false; $A\rightarrow B$ is true iff $A$ is false or $B$ is true, and false iff $A$ is true and $B$ is false; $A\leftrightarrow B$ is true iff  $A$ and $B$ are both true or both false, and false iff $A$ is true and $B$ is false or $A$ is false and $B$ is true. If $P$ is a predicate with arity $m$, then the sentence of the form $q_1x_1\dots q_mx_m P(x_1,\dots,x_m)$ is true iff the sentence $P(b_1,\dots,b_m)$ is true for all  $b_i\in N_P^i$ when $q_i$ is $\forall$, and for some $b_i\in N_P^i$ when $q_i$ is $\exists$. $q_1x_1\dots q_mx_m P(x_1,\dots,x_m)$ is  false iff the sentence $P(b_1,\dots,b_m)$ is false   for all $b_i\in N_P^i$ when $q_i$ is $\exists$, and for some $b_i\in N_P^i$ when $q_i$ is $\forall$.
\smallskip

(iii) The language is bivalent, i.e., every sentence of it is either true or false.
\smallskip

Every countable and bivalent first-order language with or without identity conforms to classical logic. A classical example is
the language of arithmetic in its standard interpretation.

We say that a language has a theory of truth if truth values are assigned to its sentences, and if it contains a predicate $T$
which satisfies
\smallskip 

 $T$-rule: \qquad \quad $T(\left\lceil A\right\rceil)\leftrightarrow A$ is true for every sentence $A$ of the language.
\smallskip

Term $\left\lceil A\right\rceil$ which names the sentence $A$ is defined below. A predicate $T$ which satisfies $T$-rule is called a truth predicate. A theory of truth 
is said to be definitional if truth values are defined for sentences, and semantical if truth values of sentences are determined by their meanings.
\smallskip

Main results of this paper are:
\smallskip

Every countable language which conforms to classical logic has an extension which has properties (i)--(iii), and has a definitional theory of truth. That extension has a semantical theory of truth if every sentence of the object language is valuated by its meaning either as true or as false. These theories of truth contain truth and  non-truth predicates. 
Theories are equivalent when the sentences of the object language are valuated by their meanings.


\section{Extended languages}\label{S2}

Assume that an object language $L_0$  conforms to classical logic, and is without a truth predicate. $L_0$ has by definition
an extension which has properties (i) -- (iii). That extension, denoted by  $L$, is called a basic extension of $L_0$.  
\smallskip

The language $L_T$ is formed by adding to $L$ extra formulas  $T(x)$ and $\neg T(x)$, and sentences $T(\bf n)$ and $\neg T(\bf n)$, where $\bf n$  goes through all numerals which denote  numbers $n\in\mathbb N_0= \{0,1,2,\dots\}$ ({\bf 0}=0, {\bf 1}=S0, {\bf 2}=SS0,\dots). Neither valuation nor meaning is yet attached to these sentences. Numerals are added, if necessary, to terms of $L_T$.
Choose a G\"odel numbering to sentences of $L_T$ (see Wikipedia). The  G\"odel number of a sentence denoted by $A$ is denoted by \#$A$, and the numeral of \#$A$ by $\left\lceil A\right\rceil$, which names the sentence $A$.
\smallskip

If $P$ is a predicate of $L$ with arity $m$, then $P(b_1,\dots b_m)$ is a sentence of $L$ for each  $(b_1,\dots b_m)\in N_P=N_1^P\times\dots\times N_P^m$, and $\left\lceil P(b_1,\dots b_m)\right\rceil$ is the numeral of its G\"odel number. Thus
$T(\left\lceil P(b_1,\dots,b_m)\right\rceil)$ and $\neg T(\left\lceil P(b_1,\dots,b_m)\right\rceil)$ are sentences of $L_T$ for each $(b_1,\dots,b_m)\in N_P$, so that  they are determined by predicates of $L_T$ having the domain $D_p$ of $P$. Denote these predicates by
$T(\left\lceil P(\dot x_1,\dots,\dot x_m)\right\rceil)$ and $\neg T(\left\lceil P(\dot x_1,\dots,\dot x_m)\right\rceil)$, and add them to $L_T$. Notation $P(\dot x_1,\dots,\dot x_m)$ stands for the result of formally replacing variables $x_1,\dots x_m$ of $P(x_1,\dots,x_m)$ by  terms of $N_P$ (cf. \cite{HH}).

Add to the language $L_T$  sentences $\forall xT_1(x)$, $\exists xT_1(x)$,
$\forall xT_1(\left\lceil T_2(\dot x)\right\rceil)$ and $\exists xT_1(\left\lceil T_2(\dot x)\right\rceil)$, where $T_1,T_2\in\{T,\neg T\}$, and sentences  $q_1x_1\dots q_mx_m T(\left\lceil P(\dot x_1,\dots,\dot x_m)\right\rceil)$  and
 $q_1x_1\dots q_mx_m \neg T(\left\lceil P(\dot x_1,\dots,\dot x_m)\right\rceil)$ for each predicate $P$ of $L$ with arity $m \ge 1$  and for each $m$-tuple $q_1,\dots,q_m$, where $q_i$'s are  $\forall$ or $\exists$.
\smallskip

The so obtained extension of the language $L_T$ is denoted by $\mathcal L_0$.
\smallskip

When a language $\mathcal L_n$, $n\in\mathbb N_0$, is defined, let $\mathcal L_{n+1}$ be a language which  is formed by adding to $\mathcal L_n$ those  of the following sentences which are not in $\mathcal L_n$: $\neg A$, $A\vee B$, $A\wedge B$, $A\rightarrow B$ and $A\leftrightarrow B$, where $A$ and $B$ are sentences of $\mathcal L_n$.
\smallskip

The language $\mathcal L$ is defined as the union of languages $\mathcal L_n$, $n\in\mathbb N_0$.
Extend the G\"odel numbering of the sentences of $L_T$ to those of $\mathcal L$, and denote by $\mathcal D$ the set of G\"odel numbers of  the sentences of $\mathcal L$.
\smallskip

Our main goal is to extract from $\mathcal L$ a sublanguage which under suitable valuations of its sentences has properties (i)--(iii) given in Introduction, and has a theory of truth.
   
 At first we define some subsets of $\mathcal L$.

Denote by $\mathcal P^m$ the set of those predicates of $L$ which have arity $m$, $\mathcal P=\bigcup_{m=1}^\infty \mathcal P^m$, and
\begin{equation}\label{E50}
\begin{cases}
Z_3^m=\{q_1x_1\dots q_mx_m T(\left\lceil P(\dot x_1,\dots,\dot x_m)\right\rceil):P\in\mathcal P^m\hbox{ and $q_1x_1\dots q_mx_m P(x_1,\dots,x_m)$ is true}\};\\
Z_4^m=\{q_1x_1\dots q_mx_m\neg T(\left\lceil P(\dot x_1,\dots,\dot x_m)\right\rceil):P\in\mathcal P^m\hbox{ and $q_1x_1\dots q_mx_m \neg P(x_1,\dots,x_m)$ is true}\}.
\end{cases}
\end{equation}
Define subsets $Z_1(U)$, $Z_2(U)$, $U\subset \mathcal D$, and  $Z_i$, $i=1\dots 4$, of $\mathcal L$ by

\begin{equation}\label{E20}
\begin{cases}
Z_1(U)=\{\hbox{$T(\bf n)$: ${\bf n}=\left\lceil A\right\rceil$, where $A$ is a sentence of $\mathcal L$ and  \#$A$  is in $U$}\},\\
Z_2(U)=\{\hbox{$\neg T(\bf n)$: ${\bf n}=\left\lceil A\right\rceil$, where $A$ is a sentence of $\mathcal L$ and  \#[$\neg A$]  is in $U$}\},\\
Z_1=\{\neg\forall x T(x),\exists xT(x),\neg\forall x\neg T(x),\exists x\neg T(x)\},\\
Z_2=\{\neg(\forall xT_1(\left\lceil T_2(\dot x)\right\rceil)), \exists xT_1(\left\lceil T_2(\dot x)\right\rceil),\ T_1,T_2\in\{T,\neg T\},\\
Z_3=\bigcup_{m=1}^\infty Z_3^m,\ Z_4=\bigcup_{m=1}^\infty Z_4^m.
\end{cases}
\end{equation}

Subsets  $L_n(U)$, $n\in\mathbb N_0$, of $\mathcal L$ are defined
 recursively as follows.

\begin{equation}\label{E201}
L_0(U)=\begin{cases} Z =\{A: A \hbox { is a true sentence of $L\}$ if $U=\emptyset$ (the empty set)},\\
    Z\cup Z_1(U)\cup Z_2(U)\cup Z_1\cup Z_2\cup Z_3\cup Z_4	\hbox{ if $\emptyset\subset U\subset \mathcal D$}.
\end{cases}
\end{equation}										

When a subset $L_n(U)$ of $\mathcal L$  is defined for some $n\in\mathbb N_0$, and when $A$ and $B$ are sentences of $\mathcal L$, denote
\begin{equation}\label{E203}
\begin{cases}
L_n^0(U)=\{\neg(\neg A):A \hbox{ is in } L_n(U)\},\\
L_n^1(U)=\{A\vee B:A  \hbox{ and $B$,  or $A$ and $\neg B$, or $\neg A$ and $B$ are in } L_n(U)\},\\
L_n^2(U)=\{A\wedge B:A \hbox{ and $B$ are in } L_n(U)\},\\
L_n^3(U)=\{A\rightarrow B:\neg A \hbox{ and $B$, or $\neg A$ and $\neg B$, or $A$ and $B$ are in } L_n(U)\},\\
L_n^4(U)=\{A\leftrightarrow B:\hbox{ $A$  and $B$,  or $\neg A$ and $\neg B$ are in } L_n(U) \},\\
L_n^5(U)=\{\neg(A\vee B):\neg A \hbox{ and $\neg B$ are in } L_n(U)\},\\
L_n^6(U)=\{\neg(A\wedge B):\neg A \hbox{ and $B$, or $\neg A$ and $\neg B$, or $A$ and $\neg B$ are in } L_n(U)\},\\
L_n^7(U)=\{\neg(A\rightarrow B):A \hbox{ and $\neg B$ are in } L_n(U)\},\\
L_n^8(U)=\{\neg(A\leftrightarrow B):\hbox{$A$ and $\neg B$, or $\neg A$ and $B$ are in } L_n(U) \},
\end{cases}
\end{equation}
and define
\begin{equation}\label{E204}
L_{n+1}(U)=L_n(U)\cup \bigcup_{k=0}^8 L_n^k(U).
\end{equation}
The above constructions imply that $L_n^k(U)\subseteq L_{n+1}^k(U)$ and
 $L_n(U)\subset L_{n+1}(U)\subset \mathcal L$  for all $n\in\mathbb N_0$ and $k=0,\dots,8$.
Define a subset $L(U)$ of $\mathcal L$ by
\begin{equation}\label{E21}
L(U)=\bigcup_{n=0}^\infty L_n(U).
\end{equation}

\section{Properties of consistent subsets of $\mathcal D$}\label{S30}

Recall that $\mathcal D$ denotes the set of G\"odel numbers of the sentences of $\mathcal L$.
When $U$ is a subset of $\mathcal D$, denote by  $G(U)$ the set of G\"odel numbers of the sentences of $L(U)$  defined by \eqref{E21}:
\begin{equation}\label{E22}
G(U)=\{\#A:A \hbox{ is a sentence of } L(U)\}.
\end{equation}

A subset $U$  of $\mathcal D$ is called consistent if  both \#$A$ and \#[$\neg A$] are not in $U$ for any sentence $A$ of $\mathcal L$ .

\begin{lemma}\label{L201} Let $U$ be a consistent subset of $\mathcal D$. Then for no sentence $A$ of $\mathcal L$ both $A$ and $\neg A$ belong to $L(U)$, and $G(U)$ is consistent.
\end{lemma}

\begin{proof} At first we show that there is no sentence $A$ in $\mathcal L$ such that both $A$ and $\neg A$ belong to $L_0(U)$.

If $U=\emptyset$, then $L_0(U)$ is  by \eqref{E201} the set $Z$ of true sentences of $L$. If  $A$ is a sentence of $L$, then only one of the sentences $A$ and $\neg A$ is true, and hence in $Z=L_0(U)$, since $L$ has properties (i)--(iii).

Assume next that $U$ is nonempty. As a consistent set $U$ is a proper subset of $\mathcal D$.

Let ${\bf n}$ be a numeral. If $T({\bf n})$ is in $L_0(U)$, it is in $Z_1(U)$, so that, by \eqref{E20}, ${\bf n}=\left\lceil A\right\rceil$, where \#$A$ is in $U$. Since $U$ is consistent, then \#[$\neg A$] is not in $U$. Thus, by \eqref{E20},
$\neg T({\bf n})$ is not in $Z_2(U)$, and hence not in $L_0(U)$. This result implies also that $T({\bf n})$ is not in $L_0(U)$ if $\neg T({\bf n})$ is in $L_0(U)$.

\eqref{E20} and \eqref{E201} imply that sentences $\exists xT(x)$, $\neg\forall xT(x)$, $\neg\forall x\neg T(x)$ and $\exists x\neg T(x)$ are in $Z_1$, and hence in $L_0(U)$, but  their negations are  not in $L_0(U)$.

By the definitions \eqref{E20} and \eqref{E201} of $Z_2$ neither both $\exists xT_1(\left\lceil T_2(\dot x)\right\rceil)$ and
 $\neg(\exists xT_1(\left\lceil T_2(\dot x)\right\rceil))$,  nor  both  $\forall xT_1(\left\lceil T_2(\dot x)\right\rceil)$ and $\neg(\forall xT_1(\left\lceil T_2(\dot x)\right\rceil))$,  are in $L_0(U)$ for any $T_1,T_2\in \{T,\neg T\}$.
By the definitions \eqref{E50}, \eqref{E20} and \eqref{E201}  the sentences of $Z_3$ and  $Z_4$, but not their negations, are in $L_0(U)$.

The above proof shows that for no sentence $A$ of $\mathcal L$ both $A$ and $\neg A$ belong to $L_0(U)$.
\smallskip

Make the induction hypothesis:  
\begin{enumerate}
\item[(h0)]
 For no sentence $A$ of $\mathcal L$ both $A$ and $\neg A$ belong to $L_n(U)$.
\end{enumerate}
Applying (h0) and \eqref{E203} we obtain the following results.

(h0) and the definition of  $L_n^0(U)$ imply that if a sentence $A$ is in $L_n(U)$, then none of the odd-tuple negations of $A$ are in $L_{n+1}(U)$, and if $\neg A$ is in $L_n(U)$, then none of the even-tuple negations  of $A$ are in $L_{n+1}(U)$.

If $A\vee B$ is in $L_{n+1}(U)$, it is in  $L_n^1(U)$, whence $A$ or $B$ is in $L_n(U)$. If $\neg(A\vee B)$ is in $L_{n+1}(U)$, it is in  $L_n^5(U)$, in which case $\neg A$ and  $\neg B$ are in $L_n(U)$. Thus $A\vee B$ and $\neg(A\vee B)$ are not both in $L_{n+1}(U)$, for otherwise  both $A$ and $\neg A$ or both  $B$ and $\neg B$ are in $L_n(U)$, contradicting with (h0).

$A\wedge B$ and $\neg(A\wedge B)$ cannot both be in $L_{n+1}(U)$, for otherwise $A\wedge B$ is in  $L_n^2(U)$, i.e., both $A$ and $B$ are in
$L_n(U)$, and $\neg(A\wedge B)$ is in  $L_n^6(U)$, i.e., at least one of $\neg A$ and $\neg B$ is in $L_n(U)$. Thus both $A$ and $\neg A$ or both  $B$ and $\neg B$ are in $L_n(U)$, contradicting with (h0).

If $A\rightarrow B$ is in $L_{n+1}(U)$, it is in  $L_n^3(U)$, so that $\neg A$ or $B$  is in $L_n(U)$. If $\neg(A\rightarrow B)$ is in $L_{n+1}(U)$, it is in $L_n^7(U)$, whence both $A$ and $\neg B$ are in $L_n(U)$. Because of these results and (h0) the sentences $A\rightarrow B$ and $\neg(A\rightarrow B)$ are not both in $L_{n+1}(U)$.

If $A\leftrightarrow B$ is $L_{n+1}(U)$, it is in $L_n^4(U)$, in which case both $A$ and $B$ or both $\neg A$ and $\neg B$ are in $L_n(U)$.
If $\neg (A\leftrightarrow B)$ is in $L_{n+1}(U)$, it is in $L_n^8(U)$, whence both $A$ and $\neg B$ or both $\neg A$ and $B$ are in $L_n(U)$. Thus both $A\leftrightarrow B$ and  $\neg (A\leftrightarrow B)$ cannot be in $L_{n+1}(U)$, for otherwise  both $A$ and $\neg A$ or both  $B$ and $\neg B$ are in $L_n(U)$, contradicting with (h0).

The above results and the induction hypothesis (h0) imply that
for no sentence $A$ of $\mathcal L$ both $A$ and $\neg A$ belong to $L_{n+1}(U)=L_n(U)\cup \bigcup_{k=0}^8 L_n^k(U)$.\\
Since (h0) is proved when $n=0$, it is by induction valid for every $n\in\mathbb N_0$.

If $A$ and $\neg A$ are in $L(U)$, then $A$ is by \eqref{E21} in $L_{n_1}(U)$ for some $n_1\in\mathbb N_0$, and  $\neg A$ is in $L_{n_2}(U)$ for some $n_2\in\mathbb N_0$. Then both $A$ and $\neg A$ are in $L_n(U)$ when $n=\max\{n_1,n_2\}$. This is impossible, because (h0) is proved for every $n\in\mathbb N_0$. Thus $A$ and $\neg A$ cannot both be in $L(U)$ for any sentence $A$ of $\mathcal L$.

The above result and \eqref{E22} imply that there is no sentence $A$ in $\mathcal L$ such that both \#$A$ and \#[$\neg A$] are in $G(U)$.
 Thus $G(U)$ is consistent.
\end{proof}

\begin{lemma}\label{L203} Assume that $U$ and $V$ are consistent  subsets of  $\mathcal D$, and that $V\subseteq U$.
 Then $L(V)\subseteq L(U)$ and $G(V)\subseteq G(U)$.
\end{lemma}

\begin{proof} As consistent sets $V$ and $U$ are proper subsets of $\mathcal D$. At first we show that $L_0(V)\subseteq L_0(U)$.

If $V=\emptyset$, then $L_0(V)=Z\subseteq L_0(U)$ by \eqref{E201}.

Assume next that $V$ is nonempty. Thus also $U$ is nonempty.

Let $A$ be a sentence of $L$. Definition \eqref{E201} of $L_0(U)$ implies that $A$ is in $L_0(U)$ and also in $L_0(V)$ iff $A$ is in $Z$.

Let {\bf n} be a numeral. If $T({\bf n})$ is in $L_0(V)$, it is  in $Z_1(V)$,  so that  ${\bf n}=\left\lceil A\right\rceil$, where \#$A$ is in $V$. Because $V\subseteq U$, then \#$A$ is also in $U$, whence $T({\bf n})$ is in $Z_1(U)$, and hence in $L_0(U)$.

If $\neg T({\bf n})$ is in $L_0(V)$, it is in $Z_2(V)$,  in which case  ${\bf n}=\left\lceil A\right\rceil$, where \#[$\neg A$] is in $V$. Since $V\subseteq U$, then \#[$\neg A$] is also in $U$, whence $\neg T({\bf n})$ is in $Z_2(U)$, and hence in $L_0(U)$.

Because $U$ and $V$ are nonempty and proper subsets of $\mathcal D$, then  $Z_1$, $Z_2$, $Z_3$ and $Z_4$ are in $L_0(U)$ and in $L_0(V)$ by \eqref{E201}.

The above results imply that $L_0(V)\subseteq L_0(U)$.
Make the induction hypothesis:
\begin{itemize}
\item[(h1)] \qquad $L_n(V)\subseteq L_n(U)$
\end{itemize}
 for some $n\in \mathbb N_0$. It follows from \eqref{E203} and (h1) that
$L_n^k(V)\subseteq L_n^k(U)$ for each $k=0,\dots,8$. Thus
$$
L_{n+1}(V)=L_n(V)\cup \bigcup_{k=0}^8 L_n^k(V)\subseteq L_n(U)\cup \bigcup_{k=0}^8 L_n^k(U)=L_{n+1}(U).
$$
(h1) is proved when $n=0$, whence it is by induction valid for every $n\in\mathbb N_0$.

If $A$ is in $L(V)$, it is by \eqref{E21} in $L_n(V)$ for some $n\in\mathbb N_0$. Thus   $A$ is in $L_n(U)$ by (h1), and hence in $L(U)$.
Consequently, $L(V)\subseteq L(U)$.

If \#$A$ is in $G(V)$ then $A$ is in $L(V)$ by \eqref{E22}. Thus $A$ is in $L(U)$, so that  \#$A$ is in $G(U)$ by \eqref{E22}.
This shows that $G(V)\subseteq G(U)$.
\end{proof}

Denote by $\mathcal C$ the family of consistent subsets of $\mathcal D$.
In the formulation and the proof of Theorem \ref{T2} transfinite sequences indexed by  ordinals are used. A transfinite sequence $(U_\lambda)_{\lambda<\alpha}$ of $\mathcal C$ is said to be increasing if
$U_\mu\subseteq U_\nu$ whenever $\mu<\nu<\alpha$, and strictly increasing if
$U_\mu\subset U_\nu$ whenever $\mu<\nu<\alpha$.

\begin{lemma}\label{L204} Assume that $(U_\lambda)_{\lambda<\alpha}$ is a strictly increasing sequence of $\mathcal C$.
Then\\
(a) \ $(G(U_\lambda))_{\lambda<\alpha}$  is an increasing sequence of $\mathcal C$.\\
(b) \ The union $\underset{\lambda<\alpha}{\bigcup}G(U_\lambda)$ is consistent.
\end{lemma}

\begin{proof}
Since  $U_\mu\subset U_\nu$ when $\mu<\nu<\alpha$, it follows from  Lemma \ref{L203} that  $G(U_\mu)\subseteq G(U_\nu)$ when $\mu<\nu<\alpha$, whence the sequence  $(G(U_\lambda))_{\lambda<\alpha}$  is increasing.
Consistency of the sets $G(U_\lambda)$, $\lambda<\alpha$, follows from Lemma \ref{L201} because the sets  $U_\lambda$, $\lambda<\alpha$, are consistent.  This proves (a).
\smallskip

To prove that the union $\underset{\lambda<\alpha}{\bigcup}G(U_\lambda)$  is consistent, assume on the contrary that there exists such a sentence $A$ in $\mathcal L$ that  both \#$A$ and \#[$\neg A$] are in $\underset{\lambda<\alpha}{\bigcup}G(U_\lambda)$.
Thus there exist $\mu,\,\nu<\alpha$ such that \#$A$ is in $G(U_\mu)$ and \#[$\neg A$] is in $G(U_\nu)$. Because
$G(U_\mu)\subseteq G(U_\nu)$ or $G(U_\nu)\subseteq G(U_\mu)$, then both  \#$A$ and \#[$\neg A$] are in $G(U_\mu)$ or in $G(U_\nu)$.
But this is impossible, since both $G(U_\mu)$ and $G(U_\nu)$ are consistent. Thus, the set $\underset{\lambda<\alpha}{\bigcup}G(U_\lambda)$  is consistent.
\end{proof}

Now we are ready to prove the following Theorem.

\begin{theorem}\label{T2} Let $W$ denote the set of G\"odel numbers of true sentences of $L$.
 We say that a transfinite sequence
$(U_\lambda)_{\lambda<\alpha}$ of $\mathcal C$ is a $G$-sequence if it has the following properties.
\begin{itemize}
\item[(G)] \ $(U_\lambda)_{\lambda<\alpha}$ is strictly increasing,
$U_0=W$, and if $0<\mu< \alpha$, then
$U_\mu = \underset{\lambda<\mu}{\bigcup}G(U_\lambda)$.
\end{itemize}
Then the longest $G$-sequence exists, and it has the last member. This member is the smallest consistent subset $U$ of $\mathcal D$  satisfying $U=G(U)$.
\end{theorem}

\begin{proof} $W$ is consistent, since $L$ has properties (i)-(iii).
At first we show that $G$-sequences are nested:
\noindent
(1) \ Assume that $(U_\lambda)_{\lambda<\alpha}$ and $(V_\lambda)_{\lambda<\beta}$ are $G$-sequences. Then $U_\lambda=V_\lambda$ when $\lambda <\min\{\alpha,\beta\}$.

 $U_0=W=V_0$ by (G). Make the induction hypothesis:\\
\smallskip
(h) \ There exists an ordinal $\nu$ which satisfies $0<\nu< \min\{\alpha,\beta\}$ such that $U_\lambda=V_\lambda$ for each $\lambda<\nu$.

It follows from (h) and (G) that $U_\nu = \underset{\lambda<\nu}{\bigcup}G(U_\lambda)=\underset{\lambda<\nu}{\bigcup}G(V_\lambda)=V_\nu$.
Since $U_0=V_0$, then (h) holds when $\nu=1$. These results imply (1) by transfinite induction.

Let $(U_\lambda)_{\lambda<\alpha}$ be a $G$-sequence. Defining  $f(0)=\min U_0$,  $f(\lambda)=\min(U_\lambda\setminus U_{\lambda-1})$, $0 < \lambda < \alpha$, and $f(\alpha)=\min(\mathcal D\setminus\underset{\lambda<\alpha}{\bigcup}U_\lambda)$, we obtain a  bijection $f$ from $[0,\alpha]$ to a subset of $\mathbb N_0$.
Thus $\alpha$ is a countable ordinal. Consequently,
 the set $\Gamma$ of those ordinals $\alpha$ for which $(U_\lambda)_{\lambda<\alpha}$ is a $G$-sequence is bounded from above by the smallest uncountable ordinal. Denote by $\gamma$ the least upper bound of $\Gamma$.

To show that $\gamma$ is a successor, assume on the contrary that $\gamma$ is a limit ordinal. Given any $\mu<\gamma$, then $\nu=\mu+1$ and $\alpha=\nu+1$ are $< \gamma$. $(U_\lambda)_{\lambda<\alpha}$ is a $G$-sequence, whence $U_\mu=\underset{\lambda<\mu}{\bigcup}G(U_\lambda)$, and  $U_\mu\subset U_{\mu+1}$. Thus
$(U_\lambda)_{\lambda<\gamma}$ has properties (G) when $\alpha=\gamma$,  so that $(U_\lambda)_{\lambda<\gamma}$ is a $G$-sequence.
  Denote $U_\gamma = \underset{\lambda<\gamma}{\bigcup}G(U_\lambda)$.  $U_\gamma$ is consistent by Lemma \ref{L204}(b). Because
 $U_\mu\subset U_\nu=\underset{\lambda<\nu}{\bigcup}G(U_\lambda)\subseteq U_\gamma$ for each $\mu<\gamma$, then $(U_\lambda)_{\lambda<\gamma+1}$ is a $G$-sequence. But this contradicts with the choice of $\gamma$.

Thus $\gamma$ is a successor, say $\gamma=\alpha+1$. If $\lambda < \alpha$, then $U_\lambda \subset U_\alpha$, so that $G(U_\lambda)\subseteq G(U_\alpha)$. Then
$U_\alpha=\underset{\lambda<\alpha}{\bigcup}G(U_\lambda)\subseteq \underset{\lambda<\gamma}{\bigcup}G(U_\lambda)=G(U_\alpha)$,
whence $U_\alpha\subseteq G(U_\alpha)$. Moreover, $U_\alpha = G(U_\alpha)$,\\
 for otherwise
$U_\alpha\subset G(U_\alpha)= \underset{\lambda<\gamma}{\bigcup}G(U_\lambda) =U_\gamma$, and  $(U_\lambda)_{\lambda<\gamma+1}$ would be a $G$-sequence.\\
Consequently,  $(U_\lambda)_{\lambda<\gamma}$ is the longest $G$-sequence, $U_\alpha$ is its last member, and $U_\alpha=G(U_\alpha)$.

Let $U$ be a consistent subset of $\mathcal D$ satisfying $U=G(U)$. Then $U_0=W=G(\emptyset)\subseteq G(U)=U$.
Make the induction hypothesis:\\
\smallskip
(h2) \ There exists an ordinal $\mu$ which satisfies $0<\mu< \gamma$ such that $U_\lambda\subseteq U$ for each $\lambda<\mu$.

Then $G(U_\lambda)\subseteq G(U)$ for each $\lambda<\mu$, whence $U_\mu=\underset{\lambda<\mu}{\bigcup}G(U_\lambda)\subseteq G(U)=U$. Thus, by transfinite induction, $U_\mu\subseteq U$ for each $\mu<\gamma$. In particular, $U_\alpha\subseteq U$. This proves the last assertion of Theorem.
\end{proof}

\section{Language $\mathcal L_T$ and its properties}\label{S3}

Let $L_0$ be a language which conforms to classical logic and has not a truth predicate. Let $\mathcal L$, $\mathcal P$ and $\mathcal D$  be as in Section \ref{S2}, and let $L(U)$ and $G(U)$, $U\subset \mathcal D$,  be defined by \eqref{E21} and \eqref{E22}. Define
\begin{equation}\label{E27}
 F(U)=\{A: \neg A\in L(U)\}.
\end{equation}
Recall that a subset $U$  of $\mathcal D$ is consistent if there is no sentence $A$ in $\mathcal L$ such that both \#$A$ and \#[$\neg A$] are in $U$.
By Theorem \ref{T2} the smallest consistent subset of $\mathcal D$ which satisfies $U=G(U)$ exists.

\begin{definition}\label{D1} Let $U$ be the smallest consistent subset of $\mathcal D$ which satisfies $U=G(U)$.
 Denote by $\mathcal L_T$ the language formed by the object language $L_0$, the predicates of $\mathcal P$, the sentences of $L(U)$ and $F(U)$, formulas $T(x)$ and $\neg T(x)$, corresponding  predicates $T$ and $\neg T$ with their domain
$D_T$ and the set $N_T$ of terms defined by
\begin{equation}\label{E28}
D_T=L(U)\cup F(U)  \hbox{ and } \
N_T=\{{\bf n}: {\bf n}=\left\lceil A\right\rceil,  \hbox{ where $A$ is in }  D_T\},
\end{equation}
and  predicates   $T_1(\left\lceil T_2(\dot x)\right\rceil)$, and $T_1(\left\lceil P(\dot x_1,\dots,\dot x_m)\right\rceil)$, where $T_1,T_2\in \cup\{T,\neg T\}$ and $P\in\mathcal P$ with arity $m\ge 1$.
\end{definition}

A valuation is defined for sentences of $\mathcal L_T$ as follows.
\begin{enumerate}
 \item[(I)] A sentence of $\mathcal L_T$ is valuated as true iff it is in $L(U)$, and as false
iff it is in $F(U)$.
\end{enumerate}

\begin{lemma}\label{L33} The language $\mathcal L_T$ defined by Definition \ref{D1} and valuated by (I) is bivalent.
\end{lemma}

\begin{proof} The subsets $L(U)$ and $F(U)$ of the sentences of $\mathcal L_T$ are disjoint. For otherwise there is a sentence $A$ of $\mathcal L_T$ which is in  $L(U)\cap F(U)$. Then $A$ is in $L(U)$, and by the definition \eqref{E27} of $F(U)$ also $\neg A$ is in $L(U)$. But this is impossible by Lemma \ref{L201}.  Consequently, $L(U)\cap F(U)=\emptyset$.
\smallskip

If $A$ is a sentence of $\mathcal L_T$, then it is in $L(U)$ or in $F(U)$.
If $A$ is true, it is in $L(U)$, but not in $F(U)$, and hence not false, because  $L(U)\cap F(U)=\emptyset$.
Similarly, if $A$ is false, it is in $F(U)$, but not in $L(U)$, and hence not true.
Consequently, $A$ is either true or false, so that $\mathcal L_T$ is bivalent.
\end{proof}

\begin{lemma}\label{L31}  Let  $\mathcal L_T$ be defined by Definition \ref{D1} and valuated by (I). Then a sentence of the basic extension $L$ of $L_0$ is true (respectively false) in the valuation (I) iff it is true (respectively false) in the valuation of $L$.
\end{lemma}

\begin{proof} Let $A$ denote a sentence of $L$. $A$ is true in the valuation (I) iff $A$ is in $L(U)$ iff (by the construction of $L(U)$) $A$ is in $Z$ iff $A$ is true in the valuation of $L$. $A$ is false in the valuation (I) iff $A$ is in $F(U)$ iff
(by \eqref{E27}) $\neg A$ is in $L(U)$ iff ($\neg A$ is a sentence of $L$) $\neg A$ is in $Z$
iff $\neg A$ is true in the valuation of $L$ iff ($L$ has properties (i)--(iii)) $A$ is false in the valuation of $L$.
\end{proof}

\begin{lemma}\label{L32} The language $\mathcal L_T$ defined by Definition \ref{D1} and valuated by (I) has properties (i) and (ii) given in Introduction.
\end{lemma}

\begin{proof} Unless otherwise stated, 'true' means true in the valuation (I), and 'false' means false in the valuation (I).

 The construction of $L(U)$ and the definition \eqref{E27} of $F(U)$ imply that $\mathcal L_T$ has properties (i).

As for properties (ii) we at first derive the following auxiliary rule.
\begin{itemize}
\item[(t0)] \ Double negation: If $A$ is a sentence of $\mathcal L_T$, then $\neg(\neg A)$ is true iff $A$ is true.
\end{itemize}
To prove (t0), assume first that $\neg(\neg A)$ is true.  Then it is in $L(U)$, and hence, by \eqref{E21}, in  $L_n(U)$ for some $n\in\mathbb N_0$. If $\neg(\neg A)$ is in $L_0(U)$ then it by \eqref{E201} in $Z$. Thus $\neg(\neg A)$ is true in the valuation of $L$. Then (negation rule is valid in $L$) $\neg A$ is false in the valuation of $L$, which implies that  $A$ is true in the valuation of $L$. Thus $A$ is by \eqref{E201} in $Z\subset L_0(U)\subset L(U)$, whence $A$ is true.\\
Assume next that  $n\in\mathbb N_0$ is the smallest number for which  $\neg(\neg A)$ is in $L_{n+1}(U)$. It then follows from \eqref{E203} and \eqref{E204} that $\neg(\neg A)$ is in  $L_{n}^0(U)$, so that $A$ is in $L_{n}(U)$, and hence in $L(U)$, i.e., $A$ is true.\\
Thus $A$ is true if $\neg(\neg A)$ is true.

Conversely, assume that  $A$ is true. Then $A$ is in $L(U)$, so that $A$  is in $L_n(U)$ for some $n\in\mathbb N_0$. Thus  $\neg(\neg A)$ is in  $L_n^0(U)$, and hence
in $L_{n+1}(U)$. Consequently, $\neg(\neg A)$ is in $L(U)$, whence $\neg(\neg A)$ is true. This concludes the proof of (t0).

Rule (t0) is applied to prove
\begin{itemize}
\item[(t1)] \ Negation: $A$ is true iff $\neg A$ is false, and $A$ is false iff $\neg A$ is true.
\end{itemize}
Let $A$ be a sentence of $\mathcal L_T$.
Then  $A$ is true iff (by (t0)) $\neg(\neg A)$ is true iff $\neg(\neg A)$ is in $L(U)$ iff (by \eqref{E27})
$\neg A$ is in $F(U)$ iff $\neg A$ is false.\\
$A$ is false iff $A$ is in $F(U)$ iff (by \eqref{E27}) $\neg A$ is in $L(U)$ iff $\neg A$ is true.
Thus (t1) is satisfied.

Next we  prove the following rule.
\begin{itemize}
\item[(t2)] \ Conjunction: $A\wedge B$ is true iff $A$ and $B$ are true. $A\wedge B$ is false iff $A$ or $B$ is false.
\end{itemize}
Let $A$ and $B$ be sentences of $\mathcal L_T$. If $A$ and $B$ are true, i.e.,  $A$ and $B$ are in $L(U)$, there is by \eqref{E21} an $n\in\mathbb N_0$ such that $A$ and $B$ are in $L_n(U)$. Thus  $A\wedge B$ is by \eqref{E203} in $L_n^2(U)$, and hence in $L(U)$, so that $A\wedge B$ is true.

Conversely, assume that $A\wedge B$ is true, or equivalently, $A\wedge B$ is in $L(U)$. Then there is by \eqref{E21} an $n\in\mathbb N_0$ such that $A\wedge B$ is in $L_n(U)$.
If $A\wedge B$ is in $L_0(U)$,  it is in $Z$.
Thus $A\wedge B$ is true in the valuation of $L$. Because $L$ has property (ii), then $A$ and $B$ are true in the valuation of $L$, and hence also in the valuation (I) by Lemma \ref{L31}.\\
 Assume next that $n\in\mathbb N_0$ is the smallest number for which  $A\wedge B$ is in $L_{n+1}(U)$. Then $A\wedge B$ is by \eqref{E203} in $L_{n}^2(U)$, so that $A$ and $B$ are in $L_{n}(U)$, and hence in
$L(U)$, i.e., $A$ and $B$ are true.

The above reasoning proves that $A\wedge B$ is true iff $A$ and $B$ are true. This result and the bivalence of $\mathcal L_T$, proved in Lemma \ref{L33}, imply that $A\wedge B$ is false iff $A$ or $B$ is false.
Consequently, rule (t2) is valid.
\smallskip

The proofs of the following rules are similar to the above proof of (t2).
\begin{itemize}
\item[(t3)] \ Disjunction: $A\vee B$  is true iff  $A$ or $B$ is true. $A\vee B$ false iff $A$ and $B$ are false.
\item[(t4)] \ Conditional: $A\rightarrow B$ is true iff $A$ is false or $B$ is true. $A\rightarrow B$ is  false iff
$A$ is true and $B$ is false.
\item[(t5)] \ Biconditional:
$A \leftrightarrow B$ is true iff $A$ and $B$ are both true or both false. $A \leftrightarrow B$ is  false iff  $A$ is true and
$B$ is false or $A$ is false and $B$ is true.
\end{itemize}
Next we  show that if $T_1\in \{T,\neg T\}$ then $\exists xT_1(x)$ and $\forall xT_1(x)$ have the following properties.
\begin{itemize}
\item[(t6)] \ \ $\exists xT_1(x)$ is true iff $T_1(\bf n)$ is true for some ${\bf n}\in N_T$, and false iff  $T_1(\bf n)$ is false for every ${\bf n} \in N_T$.
\item[(t7)] \ \   $\forall xT_1(x)$ is true iff $T_1(\bf n)$ is true for every  ${\bf n}\in N_T$, and false iff $T_1(\bf n)$ is false for some ${\bf n}\in N_T$.
\end{itemize}

To simplify proofs we  derive results which imply that $T$ is a truth predicate and $\neg T$ is a non-truth predicate for $\mathcal L_T$.\\
Let  $A$ denote a sentence of $\mathcal L_T$.
The valuation (I), rule (t1), the definitions of $Z_1(U)$, $Z_2(U)$ and $G(U)$, and the assumption  $U=G(U)$ imply that
 $A$ is  true iff $A$ is in $L(U)$ iff \#$A$ is in $G(U)=U$ iff $T(\left\lceil A\right\rceil)$ is in $Z_1(U)\subset L(U)$  iff $T(\left\lceil A\right\rceil)$ is true iff  $\neg T(\left\lceil A\right\rceil)$ is false.

 $A$ is false iff $A$ is in $F(U)$ iff $\neg A$ is in $L(U)$ iff \#[$\neg A$] is in $G(U)=U$ iff $\neg T(\left\lceil A\right\rceil)$ is in $Z_2(U)\subset L(U)$   iff $\neg T(\left\lceil A\right\rceil)$ is true iff $T(\left\lceil A\right\rceil)$ is false.
\newline
The above results imply that the following results are valid for every sentence $A\in\mathcal L_T$.
\begin{itemize}
\item[(T)] \ \ $A$ is true iff  $T(\left\lceil A\right\rceil)$ is true iff $\neg T(\left\lceil A\right\rceil)$ is false.
$A$ is false iff $T(\left\lceil A\right\rceil)$ is false iff $\neg T(\left\lceil A\right\rceil)$ is true.
\end{itemize}

Consider the validity of (t6) and (t7) when $T_1$ is $T$.
Because $U$ is nonempty, then $\exists xT(x)$ is in $L_0(U)$ by \eqref{E20} and \eqref{E201}, and hence in $L(U)$ by \eqref{E21}. Thus $\exists xT(x)$  is by (I) a true sentence of $\mathcal L_T$. \\
$T(\left\lceil A\right\rceil)$ is true iff (by (T)) $A$ is true iff (by (I)) $A$ is in $L(U)$. Thus
$T({\bf n})$ is true for some ${\bf n}\in N_T$.
\smallskip

The above results imply that $\exists xT(x)$ is true iff $T({\bf n})$ is true for some ${\bf n}\in N_T$. In view of this result and the bivalence of $\mathcal L_T$, one can infer that $\exists xT(x)$ is false iff  $T({\bf n})$ is false for every ${\bf n}\in N_T$. This concludes the proof of (p6) when $T_1$ is $T$.
\smallskip

$\neg\forall xT(x)$ is in $Z_1\subset L_0(U)$, and hence in $L(U)$, so that it is true. Thus $\forall xT(x)$ is false by (t1). \\
$T(\left\lceil A\right\rceil)$ is false iff (by (T)) $A$ is false iff (by (I)) $A$ is in $F(U)$.
 Thus  $T({\bf n})$ is false for some ${\bf n}\in N_T$.\\
Consequently,  $\forall xT(x)$ is  false iff $T({\bf n})$ is false for some ${\bf n}\in N_T$. This result and the bivalence of $\mathcal L_T$ imply
that $\forall xT(x)$ is true iff $T({\bf n})$ is true for every ${\bf n}\in N_T$. This proves (t7)  when $T_1$ is $T$.

To show that (t6) is valid when $T_1$ is $\neg T$, notice first that  $\exists x\neg T(x)$ is in $Z_1\subset L_0(U)$, and hence in $L(U)$, whence it is true.\\
$\neg T(\left\lceil A\right\rceil)$ is true iff (by (t1)) $T(\left\lceil A\right\rceil)$ is false iff (by (T)) $A$ is false iff (by (I)) $A$ is in $F(U)$.
 Thus  $\neg T({\bf n})$ is true for some  ${\bf n}\in N_T$.
Consequently, $\exists x\neg T(x)$ is true iff $\neg T({\bf n})$ is true for some  ${\bf n}\in N_T$. This result and the bivalence of $\mathcal L_T$ imply that $\exists x\neg T(x)$ is  false iff $\neg T({\bf n})$ is false for every ${\bf n}\in N_T$. This concludes the proof of (t6) when $T_1$ is $\neg T$.
\smallskip

Next we  prove (t7) when $T_1$ is $\neg T$.
$\neg\forall x\neg T(x)$ is in $Z_1\subset L_0(U)$, and hence in $L(U)$, so that it is true. Thus $\forall x\neg T(x)$ is false by (t1). \\
$\neg T(\left\lceil A\right\rceil)$ is false iff (by (t1)) $T(\left\lceil A\right\rceil)$ is true iff (by (T)) $A$ is true iff (by (I)) $A$ is in $L(U)$.
Thus  $\neg T({\bf n})$ is false for some ${\bf n}\in N_T$.
From these results it follows that $\forall x\neg T(x)$ is false iff $\neg T({\bf n})$ is false for some ${\bf n}\in N_T$. This result and bivalence of $\mathcal L_T$ imply that $\forall x\neg T(x)$ is true iff $\neg T({\bf n})$ is true for all ${\bf n}\in N_T$. Thus (t7) is valid when $T_1$ is $\neg T$.

Next we show that the following rules are valid when $T_1,T_2\in \{T,\neg T\}$.
\begin{itemize}
\item[(tt6)] \ \ \ \     $\exists xT_1(\left\lceil T_2(\dot x)\right\rceil)$ is true iff $T_1(\left\lceil T_2(\bf n)\right\rceil)$ is true for some ${\bf n}\in N_T$.
\item[]\quad $\exists xT_1(\left\lceil T_2(\dot x)\right\rceil)$ is false iff  $T_1(\left\lceil T_2(\bf n)\right\rceil)$ is false for every ${\bf n}\in N_T$;
\item[(tt7)] \ \ \ \      $\forall xT_1(\left\lceil T_2(\dot x)\right\rceil)$ is true iff $T_1(\left\lceil T_2(\bf n)\right\rceil)$ is true for every  ${\bf n}\in N_T$.
\item[]\quad $\forall xT_1(\left\lceil T_2(\dot x)\right\rceil)$ is false iff
$T_1(\left\lceil T_2(\bf n)\right\rceil)$ is false for some ${\bf n}\in N_T$.
\end{itemize}

The sentences $\exists xT_1(\left\lceil T_2(\dot x)\right\rceil)$, where $T_1,T_2\in\{T,\neg T\}$ are in $Z_2$, whence they are in $L(U)$ and hence true. Applying (T) one can  show that in every case there exists an ${\bf n}\in N_T$ so that $T_1(\left\lceil T_2(\bf n)\right\rceil)$ is true (${\bf n}=\left\lceil A\right\rceil$, where $A$, depending on the case,  is in $L(U)$ or in $F(U)$). These results imply truth part of (tt6) when  $T_1$ and $T_2$ are in $\{T,\neg T\}$. Falsity part in (tt6) is then valid by bivalence of $\mathcal L_T$.
\smallskip

$\forall xT_1(\left\lceil T_2(\dot x)\right\rceil)$ is false because its negation is in $Z_2$ and hence true. Using (T) it is easy to show that
$T_1(\left\lceil T_2(\bf n)\right\rceil)$ is false for some ${\bf n}\in N_T$ whenever $T_1,T_2\in\{T,\neg T\}$. This proves the falsity  part of (tt7), and also implies the truth part   by bivalence of $\mathcal L_T$.
\smallskip

Since $L$ has properties (ii), then for every predicate $P\in \mathcal P^m$  with arity $m\ge 1$ the following properties hold in the valuation of $L$, and hence in the valuation  (I) by Lemma \ref{L31}.
\begin{itemize}
\item[(p6)]  The sentence of the form $q_1x_1\dots q_mx_m P(x_1,\dots,x_m)$ is true iff  $P(b_1,\dots,b_m)$ is true for all  $b_i\in N_P^i$ when $q_i$ is $\forall$, and for some $b_i\in N_P^i$ when $q_i$ is $\exists$. $q_1x_1\dots q_mx_m P(x_1,\dots,x_m)$ is  false iff  $P(b_1,\dots,b_m)$ is false for all $b_i\in N_P^i$ when $q_i$ is $\exists$, and for some $b_i\in N_P^i$ when $q_i$ is $\forall$.
\end{itemize}
If $P$ is a predicate of $\mathcal P$ with arity $m\ge 1$, and if $q_1,\dots,q_m$ is any  $m$-tuple of quantifiers $\forall$ and $\exists$,  then the sentence $q_1x_1\dots q_mx_m T(\left\lceil P(\dot x_1,\dots,\dot x_m)\right\rceil)$ is true iff  it is in $L_0(U)$ iff it is in $Z_1^m$ iff (by \eqref{E50}) the sentence $q_1x_1\dots q_mx_m P(x_1,\dots,x_m)$ is true iff (by (p6)) the sentence $P(b_1,\dots,b_m)$ is true in $L$, and hence also in $\mathcal L_T$  for all $b_i\in N_P^i$ when $q_i$ is $\forall$, and for some  $b_i\in N_P^i$ when $q_i$ is $\exists$ iff (by (T)) the sentence $T(\left\lceil P(b_1,\dots,b_m)\right\rceil)$ is true  for all choices of $b_i\in N_P^i$ when $q_i$ is $\forall$, and for some choices of $b_i\in N_P^i$ when $q_i$ is $\exists$.
\smallskip

The above equivalences and the bivalence of $\mathcal L_T$ imply the following result.
\begin{itemize}
\item[(tp6)]\    The sentence $q_1x_1\dots q_mx_m T(\left\lceil P(\dot x_1,\dots,\dot x_m)\right\rceil)$ true iff  the sentence $T(\left\lceil P(b_1,\dots,b_m)\right\rceil)$ is true  for all choices of $b_i\in N_P^i$ when $q_i$ is $\forall$, and for some choices of $b_i\in N_P^i$ when $q_i$ is $\exists$.
\item[]\ $q_1x_1\dots q_mx_m T(\left\lceil P(\dot x_1,\dots,\dot x_m)\right\rceil)$ is  false iff the sentence $T(\left\lceil P(b_1,\dots,b_m)\right\rceil)$ is false for all $b_i\in N_P^i$ when $q_i$ is $\exists$, and for some $b_i\in N_P^i$ when $q_i$ is $\forall$.
\end{itemize}

Similarly, applying \eqref{E50}, (T), (p6) with $P$ replaced by $\neg P$ and bivalence of $\mathcal L_T$ it one can show that $T$ can be replaced in  (tp6) by $\neg T$  when  $P$ is in $\mathcal P^m$ and  $q_1,\dots,q_m$ is any  $m$-tuple of quantifiers $\forall$ and $\exists$.
\smallskip

The above proof shows that $\mathcal L_T$ has properties (ii).
\end{proof}


\section{Theories of truth}\label{S4}

Now we are ready to present our main results.
\smallskip

\begin{theorem}\label{T1} Let $L_0$ be a countable language which conforms to classical logic and has not a truth predicate. The language $\mathcal L_T$ defined in Definition \ref{D1} and valuated by (I) has properties (i)--(iii), and has a definitional theory of truth (shortly DTT). $T$ is a truth predicate, and $\neg T$ is a non-truth predicate. 
\end{theorem}

\begin{proof} Properties (i)--(iii) given in Introduction are valid for $\mathcal L_T$ by
 Lemma \ref{L31}  and Lemma \ref{L32}. Thus $\mathcal L_T$ conforms to classical logic.
 \smallskip

The results (T) derived in the proof of Lemma \ref{L32} and biconditional rule (t5) imply that the sentence
$T(\left\lceil A\right\rceil)\leftrightarrow A$ is true and the sentence $\neg T(\left\lceil A\right\rceil)\leftrightarrow A$ is false for every sentence $A$ of $\mathcal L_T$.
  $T$ and $\neg T$ are predicates of $\mathcal L_T$, and their domain $D_T$, the set
all sentences of $\mathcal L_T$, satisfies the condition presented in \cite[p. 7]{Fe} for the domains of truth predicates. Consequently, $T$ is a truth predicate and $\neg T$ is a non-truth predicate.
The above results imply that $\mathcal L_T$ has a theory of truth. It is definitional, since truth values of sentences are defined by (I).
\end{proof}
\smallskip

Next we  show that $\mathcal L_T$ has  a semantical theory of truth
under the following assumptions.
\begin{itemize}
\item[(s1)] \ \ The object language $L_0$ is countable, has not a truth predicate, and every sentence of $L_0$  is

$ $ \ \   meaningful and is valuated by its meaning either as true or as false.
\item[(s2)] \ \ Standard meanings are assigned to logical symbols.
\item[(s3)] \ \ The sentence $T(\bf n)$ means: 'the sentence whose G\"odel number has $\bf n$ as its numeral is true'.
\end{itemize}

 At first we prove preliminary Lemmas.

\begin{lemma}\label{L42} Under the hypotheses (s1) and (s2) the object language $L_0$ has an extension $L$ which has properties (i)--(iii) given in Introduction when its sentences are valuated by their meanings. In particular, $L_0$ conforms to classical logic. 
\end{lemma}

\begin{proof} The object language $L_0$ is bivalent by (s1). 
The basic extension $L$ of $L_0$ which has properties (i)--(iii) when its sentences are valuated by their meanings is constructed as follows.
The first extension $L_1$ of $L_0$ is formed  by adding those sentences  $\neg A$, $A\vee B$, $A\wedge B$, $A\rightarrow B$, $A\leftrightarrow B$ and $q_1x_1\dots q_mx_m P(x_1,\dots,x_m)$ which are not in $L_0$ when $A$ and $B$ go through all sentences of $L_0$, $P$ its predicates and their negations (added if necessary), and $(q_1,\dots,q_m)$ $m$-tuples, where 
$q_i$'s are  either $\forall$ or $\exists$. If there exist objects of $D$ without names, add terms to name them. Valuating the sentences of $L_1$ by their meanings, the assumptions (s1) and (s2) ensure that properties (ii) and (iii) are valid.

When the language  $L_n$, $n\ge 1$ is defined, define the language $L_{n+1}$ by adding those sentences  $\neg A$, $A\vee B$, $A\wedge B$, $A\rightarrow B$, $A\leftrightarrow B$ which are not in $L_n$ when $A$ and $B$ are sentences of $L_n$. Valuating the sentences of $L_{n+1}$ by their meanings, the properties (ii) and (iii) are valid by assumptions (s1) and (s2). The union $L$ of languages $L_n$, $n\in \mathbb N_0$ has also properties (ii) and (iii).

If $A$ and $B$ denote sentences of $L$, there exist $n_1$ and $n_2$ in $\mathbb N_0$ such that $A$ is in $L_{n_1}$ and $B$ is in $L_{n_2}$.
Denoting $n=\max\{n_1,n_2\}$, then $A$ and $B$ are sentences of $L_n$. Thus the sentences $\neg A$, $A\vee B$, $A\wedge B$, $A\rightarrow B$ and $A\leftrightarrow B$ are in $L_{n+1}$, and hence in $L$. If $P$ is a predicate of $L_1$ with arity $m\ge 1$, then for each $m$-tuple $(q_1,\dots,q_m)$, where 
$q_i$'s are  either $\forall$ or $\exists$, the sentence $q_1x_1\dots q_mx_m P(x_1,\dots,x_m)$ is in $L_1$, so that it is in $L$. Thus $L$ has also  properties (i).

Consequently, the basic extension $L$ of $L_0$ constructed above has all properties (i) -- (iii) when its sentences are valuated by their meanings. This implies that $L_0$ conforms to classical logic.
\end{proof}

 \begin{lemma}\label{l70} Make the assumptions (s1)--(s3), and let $L$ be the language constructed in the proof of Lemma \ref{L42}. Let $\mathcal L_T$, $\mathcal D$ and $U$ be as in Definition \ref{D1}, and let $W$ be the set of G\"odel numbers of true sentences of $L$. Given a consistent subset $V$ of $\mathcal D$ which satisfies $W\subseteq V\subseteq U$, assume that every sentence of $\mathcal L_T$ whose G\"odel number is in $V$ is true and not false by its meaning.
Then every sentence of $L(V)$ (defined by \eqref{E21} with $U$ replaced by $V$) is true and not false by its meaning.
\end{lemma}

\begin{proof} Because $V\subseteq U=G(U)$, then every sentence whose G\"odel number is in $V$, is in $\mathcal L_T$.
At first we prove that every sentence of $L_0(V)$ is true and not false by its meaning.
\smallskip

By Lemma \ref{L42} $L$ is bivalent. Thus every true sentence of $L$, i.e., every sentence of $Z$ is true and not false by its meaning.

 Let $A$ denote a sentence of $\mathcal L_T$.
  By (s3) the sentence  $T(\left\lceil A\right\rceil)$ means that 'the sentence whose G\"odel number has $\left\lceil A\right\rceil$ as its numeral, i.e., the sentence $A$, is true'. Thus, by its meaning,  $T(\left\lceil A\right\rceil)$   is true iff $A$ is true and false iff $A$ is false.
	
Assume that the G\"odel number of a sentence $A$ is in $V$.
Since $A$ is by a hypothesis true and not false by its meaning,  then $T(\left\lceil A\right\rceil)$ is true and not false by its meaning.
This implies by \eqref{E20} that the sentences of $Z_1(V)$ are true and not false by their meanings.
By the standard meaning of negation the sentence $\neg T(\left\lceil A\right\rceil)$ is false and not true by its meaning. Replacing $A$ by $T(\left\lceil A\right\rceil)$ and $\neg T(\left\lceil A\right\rceil)$ it follows from the above results that the sentences $T(\left\lceil T(\left\lceil A\right\rceil)\right\rceil)$ and $\neg T(\left\lceil \neg T(\left\lceil A\right\rceil)\right\rceil)$ are true and not false by their meanings, and the sentences
$T(\left\lceil \neg T(\left\lceil A\right\rceil)\right\rceil)$  and $\neg T(\left\lceil T(\left\lceil A\right\rceil)\right\rceil)$ are false and not true
by their meanings.

Let  $A$ denote such a sentence
of $\mathcal L_T$, that  the G\"odel number of the sentence $\neg A$ is in $V$.
$\neg A$ is by a hypothesis  true and not false by its meaning, so that $A$ is false and not true by its meaning since $V$ is consistent. Thus the sentence $T(\left\lceil A\right\rceil)$ is false and not true by its meaning, and the sentence $\neg T(\left\lceil A\right\rceil)$
is true and not false by its meaning. It then follows from  \eqref{E20} that the sentences of $Z_2(V)$ are true and not false by their meanings.

Replacing $A$ by $T(\left\lceil A\right\rceil)$ and $\neg T(\left\lceil A\right\rceil)$ we then obtain that the sentences $T(\left\lceil T(\left\lceil A\right\rceil)\right\rceil)$  and $\neg T(\left\lceil \neg T(\left\lceil A\right\rceil)\right\rceil)$ are  false and not true by their meanings, and  the sentences  $\neg T(\left\lceil T(\left\lceil A\right\rceil)\right\rceil)$  and $T(\left\lceil \neg T(\left\lceil A\right\rceil)\right\rceil)$ are true and not false by their meanings.

The set $N_T$ of numerals, defined by \eqref{E28}, is formed by numerals $\left\lceil A\right\rceil$, where  $A$ goes through all the sentences of $\mathcal L_T$.  Thus, by results proved above
$T({\bf n})$, $T(\left\lceil T({\bf n})\right\rceil)$, $\neg T(\left\lceil\neg T({\bf n})\right\rceil)$, $\neg T({\bf n})$, $\neg T(\left\lceil T({\bf n})\right\rceil)$ and $T(\left\lceil\neg T({\bf n})\right\rceil)$ are  for some ${\bf n}\in N_T$  true and not false  by their meanings and for some ${\bf n}\in N_T$ false and not true by their meanings.
These results and the standard meanings of quantifiers and negation imply that $\exists x T(x)$, $\exists xT(\left\lceil T(\dot x)\right\rceil)$, $\exists x\neg T(\left\lceil\neg T(\dot x)\right\rceil)$, $\exists x \neg T(x)$, $\exists x\neg T(\left\lceil T(\dot x)\right\rceil)$ and $\exists xT(\left\lceil\neg T(\dot x)\right\rceil)$ are  true and not false by their meanings, and their negations are false and not true by their meanings, whereas $\forall x T(x)$, $\forall xT(\left\lceil T(\dot x)\right\rceil)$, $\forall x\neg T(\left\lceil\neg T(\dot x)\right\rceil)$, $\forall x \neg T(x)$, $\forall x\neg T(\left\lceil T(\dot x)\right\rceil)$ and $\forall xT(\left\lceil\neg T(\dot x)\right\rceil)$ are false and not true by their meanings, and their negations are true and not false by their meanings.

By above results and \eqref{E20} the sentences of $Z_1$ and  $Z_2$ are  true and not false by their meanings.

Let $P$ be a predicate in $\mathcal P^m$, and let $q_1,\dots,q_m$ be an  $m$-tuple of quantifiers $\forall$ and $\exists$.
Since the sentences $q_1x_1\dots q_mx_m P(x_1,\dots,x_m)$ and $P(b_1,\dots,b_m)$ are in $L$, they are by their meanings either true and not false,  or false and not true.

If $q_1x_1\dots q_mx_m P(x_1,\dots,x_m)$ is true and not false by its meaning, then
$P(b_1,\dots,b_m)$ is true and not false by its meaning for all $b_i$ when $q_i$ is $\forall$, and for some $b_i$ when $p_i$ is $\exists$. Thus, by (s3), $T(\left\lceil P(b_1,\dots,b_m)\right\rceil)$ is true and not false by its meaning for all $b_i$ when $q_i$ is $\forall$, and for some $b_i$ when $p_i$ is $\exists$. Consequently,  $q_1x_1\dots q_mx_m T(\left\lceil P(b_1,\dots,b_m)\right\rceil)$ is true and not false by its meaning.

By its meaning $\neg P(b_1,\dots,b_m)$ is true and not false iff $P(b_1,\dots,b_m)$ is false and not true iff, by (s3), $T(\left\lceil P(b_1,\dots,b_m)\right\rceil)$ is false and not true iff $\neg T(\left\lceil P(b_1,\dots,b_m)\right\rceil)$ is true and not false.
Thus the sentence $q_1x_1\dots q_mx_m \neg T(\left\lceil P(\dot x_1,\dots,\dot x_m)\right\rceil)$  is true and not false by its meaning iff the sentence $q_1x_1\dots q_mx_m \neg P(x_1,\dots,x_m)$ is true and not false by its meaning.

It follows from the above results and \eqref{E50} that the sentences of $Z_1^m$ and $Z_2^m$ are true and not false by their meanings.
Consequently, the sentences of $Z_3$ and $Z_4$ are by \eqref{E20} true and not false by their meanings.

The above results and \eqref{E201} imply that every sentence of $L_0(V)$ is true and not false by its meaning. Thus the following property holds when $n=0$.
\begin{enumerate}
\item[(h3)] \ \ Every sentence of $L_n(V)$ is true and not false by its meaning.
\end{enumerate}
Make the induction hypothesis: (h3) holds for some $n\in\mathbb N_0$.

Given a sentence of  $L_n^0(V)$, it is of the form $\neg(\neg A)$, where  $A$ is in $L_n(V)$.
$A$ is by (h3) true and not false by its meaning. Thus, by standard meaning of negation, its double application implies that the sentence $\neg(\neg A)$,  and hence the given sentence, is true and not false by its meaning.

If a sentence is in $L_n^1(V)$, it is of the form $A\vee B$, where  $A$ or $B$ is in $L_n(V)$.
By (h3) at least one of the sentences  $A$ and  $B$ is true and not false by its meaning.
 Thus the sentence
$A\vee B$, and hence  given sentence, is true and not false by its meaning.

 Similarly it can be shown that if (h3) holds, then every sentence of $L_n^k(V)$, where  $2 \le k\le 8$, is true and not false by its meaning.

The above results imply that under the induction hypothesis (h3)
every sentence of  $L_n^k(V)$, where  $0 \le k\le 8$, is true and not false by its meaning.

It then follows from the definition \eqref{E204} of $L_{n+1}(V)$ that if (h3) is valid for some $n\in\mathbb N_0$,
then every sentence of $L_{n+1}(V)$ is true and not false by its meaning.

The first part of this proof shows that (h3) is valid when $n=0$. Thus, by induction,
it is valid for all $n\in\mathbb N_0$.
This result and  \eqref{E21} imply that  every sentence of $L(V)$ is true and not false by its meaning.
\end{proof}

\begin{lemma}\label{L9} Let $U$ be the smallest consistent subset of $\mathcal D$ which satisfies $U=G(U)$. Then under the hypotheses (s1)--(s3) every sentence  of  $\mathcal L_T$ whose G\"odel number is in $U$ is true and not false by its meaning.
\end{lemma}

\begin{proof} By Theorem \ref{T2} the smallest consistent subset $U$ of $\mathcal D$ which satisfies $U=G(U)$ is the last member of the transfinite sequence $(U_\lambda)_{\lambda<\gamma}$  constructed in the proof of that Theorem.
We prove by transfinite induction that the following result holds for all $\lambda < \gamma$.
\begin{enumerate}
\item[(H)] Every sentence  of  $\mathcal L_T$ whose G\"odel number is in $U_\lambda$ is true and not false by its meaning.
\end{enumerate}
Make the induction hypothesis: There is a $\mu$ satisfying $0<\mu< \gamma$ such that (H) holds for all $\lambda < \mu$.

Let $\lambda < \mu$ be given. Because $U_\lambda$ is consistent and $W\subseteq U_\lambda\subseteq U$ for every $\lambda < \mu$, it follows from the induction hypothesis and Lemma \ref{l70} that every sentence of $L(U_\lambda)$ is true and not false by its meaning. This implies by \eqref{E22} that
(H) holds  when $U_\lambda$ is replaced $G(U_\lambda)$, for every  $\lambda <\mu$. Thus (H) holds when $U_\lambda$ is replaced by the union of those sets. But this union is $U_\mu$ by Theorem \ref{T2} (G), whence (H) holds when $\lambda =\mu$.

When $\mu =1$, then $\lambda<\mu$ iff $\lambda=0$.
$U_0=W$, i.e., the set of G\"odel numbers of true sentences of $L$. Since  $L$, valuated by meanings of its sentences, is bivalent by Lemma \ref{L42}, the sentences of $U_0$ are true and not false by their meanings.
This proves that the induction hypothesis is satisfied  when $\mu=1$.

The above proof implies by transfinite induction properties assumed in (H) for $U_\lambda$ whenever
 $\lambda <\gamma$. In particular the last member of   $(U_\lambda)_{\lambda<\gamma}$ satisfies (H),
which is by Theorem \ref{T2} the smallest consistent subset $U$  of $\mathcal D$ for which $U=G(U)$.
 This  proves the assertion.
\end{proof}

The  next result is  a consequence of Lemma \ref{L42}, Lemma \ref{L9} and Theorem \ref{T1}.

\begin{theorem}\label{T0} Under the hypotheses (s1)--(s3) the extension $\mathcal L_T$ of $L_0$ defined in Definition \ref{D1} has  a semantical theory of truth (shortly STT), when the valuation (I) is replaced in Theorem \ref{T1} with the valuation of the sentences of $\mathcal L_T$ by their  meanings. This valuation is equivalent to  valuation (I), and the results of Theorem \ref{T1} are valid for STT.
\end{theorem}

\begin{proof}
Let $A$ denote a sentence of $\mathcal L_T$. $A$ is by Definition \ref{D1} in $L(U)$ or in $F(U)$, where $U$ is the smallest consistent subset of $\mathcal D$ which satisfies $U=G(U)$.
 If $A$ is in $L(U)$, its G\"odel number is in $G(U)=U$ by \eqref{E22}, whence it is by Lemma \ref{L9} true and not false by its meaning.
If $A$ is in $F(U)$, then $\neg A$ is in $L(U)$ by \eqref{E27}. Thus $\neg A$ is true and not false by its meaning, so that $A$ is by the standard meaning of negation false and not true by its meaning. Hence the valuation of $\mathcal L_T$ by meanings of its sentences is equivalent to valuation (I). In particular, $\mathcal L_T$ has properties (i)-(iii). Moreover, the sentence $T(\left\lceil A\right\rceil)\leftrightarrow A$ is true by its meaning and the sentence $\neg T(\left\lceil A\right\rceil)\leftrightarrow A$ is false by its meaning for every sentence $A$ of $\mathcal L_T$.
These results imply that $T$ is a truth predicate and $\neg T$ is a non-truth predicate for $\mathcal L_T$. Consequently, $\mathcal L_T$  
has a semantical theory of truth. 

\end{proof}


\section{Compositionality of truth in theories DTT and STT}\label{S6}

 One of the desiderata introduced in \cite{Lei07} for theories of truth is that truth should be compositional. In this section we  present some logical equivalences which theories DTT and STT of truth prove.

\begin{lemma}\label{L41} Theories DTT and STT of truth presented in Theorems \ref{T1} and \ref{T0} prove the following
 logical equivalences  when $A$ and $B$ are sentences of $\mathcal L_T$, and $P$ is in $\mathcal P$ or $P$ is $T$.
\item[(a0)] \ \ $T(\left\lceil\dots \left\lceil T(\left\lceil A\right\rceil)\right\rceil\dots\right\rceil)\leftrightarrow T(\left\lceil T(\left\lceil A\right\rceil)\right\rceil)\leftrightarrow T(\left\lceil A\right\rceil)\leftrightarrow A$.
\item[(a1)] \ \ $\neg T(\left\lceil A\right\rceil)\leftrightarrow T(\left\lceil \neg A\right\rceil)\leftrightarrow \neg A$.
\item[(a2)] \ \ $T(\left\lceil A\right\rceil)\vee T(\left\lceil B\right\rceil)\leftrightarrow T(\left\lceil A\vee B\right\rceil)\leftrightarrow A\vee B$.
\item[(a3)] \ \ $T(\left\lceil A\right\rceil)\wedge T(\left\lceil B\right\rceil) \leftrightarrow T(\left\lceil A\wedge B\right\rceil) \leftrightarrow  A\wedge B$.
\item[(a4)] \ \ $(T(\left\lceil A\right\rceil)\rightarrow T(\left\lceil B\right\rceil)) \leftrightarrow T(\left\lceil A\rightarrow B\right\rceil) \leftrightarrow  (A \rightarrow B)$.
\item[(a5)] \ \ $(T(\left\lceil A\right\rceil)\leftrightarrow T(\left\lceil B\right\rceil)) \leftrightarrow T(\left\lceil A\leftrightarrow B\right\rceil) \leftrightarrow  (A \leftrightarrow B)$.
\item[(a6)] \ \ $\neg T(\left\lceil A\vee B\right\rceil)\leftrightarrow \neg (A\vee B)\leftrightarrow \neg A\wedge\neg B\leftrightarrow T(\left\lceil\neg A\right\rceil)\wedge T(\left\lceil\neg B\right\rceil)\leftrightarrow\neg T(\left\lceil A\right\rceil)\wedge\neg T(\left\lceil B\right\rceil)$.
\item[(a7)] \ \ $\neg T(\left\lceil A\wedge B\right\rceil)\leftrightarrow \neg (A\wedge B)\leftrightarrow \neg A\vee\neg B\leftrightarrow T(\left\lceil\neg A\right\rceil)\vee T(\left\lceil\neg B\right\rceil)\leftrightarrow \neg T(\left\lceil A\right\rceil)\vee\neg T(\left\lceil B\right\rceil)$.
\item[(a8)] \ \ $\forall xT(\left\lceil P(\dot x)\right\rceil)\leftrightarrow T(\left\lceil \forall xP(x)\right\rceil)\leftrightarrow \forall xP(x)\leftrightarrow \neg \exists x \neg P(x)\leftrightarrow T(\left\lceil\neg\exists x \neg P(x)\right\rceil)\leftrightarrow\neg \exists x \neg T(\left\lceil P(\dot x)\right\rceil)$.
\item[(a9)] \ \ $\exists xT(\left\lceil P(\dot x)\right\rceil) \leftrightarrow T(\left\lceil \exists xP(x)\right\rceil)\leftrightarrow \exists xP(x)\leftrightarrow\neg \forall x \neg P(x)\leftrightarrow T(\left\lceil\neg\forall x \neg P(x)\right\rceil)\leftrightarrow\neg \forall x \neg T(\left\lceil P(\dot x)\right\rceil)$.
\item[(a10)] \ $\neg T(\left\lceil \forall xP(x)\right\rceil)\leftrightarrow T(\left\lceil\neg \forall x (P(x)\right\rceil)\leftrightarrow\neg \forall x P(x)\leftrightarrow \exists x \neg P(x)\leftrightarrow T(\left\lceil\exists x \neg P(x)\right\rceil)$.
\item[(a11)] \ $\neg T(\left\lceil \exists xP(x)\right\rceil) \leftrightarrow T(\left\lceil\neg \exists x P(x)\right\rceil)\leftrightarrow \neg \exists x P(x)\leftrightarrow \forall x \neg P(x)\leftrightarrow T(\left\lceil\forall x \neg P(x)\right\rceil)$.
\end{lemma}

\begin{proof} T-rule implies equivalences of (a0).

The first equivalences in (a1)--(a5) are easy consequences of rules (t1)--(t5) and $T$-rule (cf. \cite[Lemma 4.1]{Hei18}). Their second equivalences are consequences of $T$-rule.

The first and last equivalences of (a6) and (a7) are consequences of (a1).
 Their second equivalences are De Morgan's laws of classical logic (cf. \cite{1}).  Third equivalences of (a6) and (a7) follow from $T$-rule.
The first equivalences of (a8) and (a9) are easy consequences of rules (tp6) and (tp7) and $T$-rule (cf. \cite[Lemma 4.2]{Hei18}). $T$-rule implies their second equivalences.
The third equivalences are  De Morgan's laws for quantifiers (cf. \cite{1}). The fourth ones follow from $T$-rule. De Morgan's laws with $P(x)$ replaced by $T(\left\lceil P(\dot x)\right\rceil)$ imply equivalence of the last and the first sentences.
(a10) and (a11) are negations to some equivalences of (a8) and (a9).
\end{proof}

Let $L_0$ be a countable and bivalent first-order language with or without identity. It conforms to classical logic.
If $P$ and $Q$ are predicates of $L_0$ with arity 1 and domain $D$, then $\neg P$, $P\vee Q$, $P\wedge Q$, $P\rightarrow Q$ and
$P\leftrightarrow Q$ are predicates of $L_0$ with domain $D$. Replacing $P$ and/or $Q$ by some of them we obtain new predicates with domain $D$, and so on. Thus $P$ in (a8) and (a9) can be replaced by anyone of these predicates.  Their universal and existential quantifications are  sentences of  $L$. They are also sentences of $\mathcal L_T$. Anyone of them can be the sentence $A$ and/or the sentence $B$
in results (a1)--(a7) derived above. Moreover, $P$ can be replaced by anyone of those predicates in (a8)--(a11). Take a few examples.
\smallskip

$\forall x T(\left\lceil P(\dot x)\rightarrow Q(\dot x)\right\rceil) \leftrightarrow T(\left\lceil \forall x (P(x)\rightarrow Q(x))\right\rceil)\leftrightarrow \forall x (P(x)\rightarrow Q(x))$.

$\exists x T(\left\lceil P(\dot x)\wedge Q(\dot x)\right\rceil)\leftrightarrow T(\left\lceil\exists x(P(x)\wedge Q(x)\right\rceil)\leftrightarrow\exists x(P(x)\wedge Q(x))$.

$\forall x T(\left\lceil P(\dot x)\rightarrow \neg Q(\dot x)\right\rceil) \leftrightarrow T(\left\lceil \forall x (P(x)\rightarrow \neg Q(x))\right\rceil)\leftrightarrow \forall x (P(x)\rightarrow \neg Q(x))$.

$\exists xT(\left\lceil P(\dot x)\wedge\neg Q(\dot x)\right\rceil)\leftrightarrow T(\left\lceil\exists x(P(x)\wedge \neg Q(x)\right\rceil)\leftrightarrow\exists x(P(x)\wedge\neg Q(x))$.
\smallskip

These equivalences   correspond to the four Aristotelian forms: 'All $P$'s are $Q$'s', 'some $P$'s are $Q$'s',
 'no $P$'s are $Q$'s' and 'some $P$'s are not $Q$'s' (cf. \cite{1}).
\smallskip

Let $P$ be a predicate of $L_0$ with arity $m>1$, and let $q_1,\dots q_m$ be any of the $2^m$ different $m$-tuples which can be formed from quantifiers $\forall$ and $\exists$. Theories DTT  and STT presented in Theorems \ref{T1} and \ref{T0} prove the following logical equivalences.\\
$T(\left\lceil q_1 x_1\dots q_m x_mP(x_1,\dots,x_m)\right\rceil)\leftrightarrow q_1 x_1\dots q_m x_m P(x_1,\dots,x_m)\leftrightarrow q_1 x_1\dots q_m x_mT(\left\lceil P((\dot x_1,\dots,\dot x_m)\right\rceil)$,\\
$T(\left\lceil q_1 x_1\dots q_m x_m\neg P(x_1,\dots,x_m)\right\rceil)\leftrightarrow q_1 x_1\dots q_m x_m\neg P(x_1,\dots,x_m)\leftrightarrow q_1 x_1\dots q_m x_m\neg T(\left\lceil P(\dot x_1,\dots,\dot x_m)\right\rceil)$.
 $T$-rule implies the first equivalences, and the second equivalences are consequences of \eqref{E50} and bivalence of $\mathcal L_T$.
 An application of $T$-rule proves
the universal $T$-schema: \\
(UT)\qquad
$\forall x_1\dots\forall x_m\big{(}T(\left\lceil P((\dot x_1,\dots,\dot x_m)\right\rceil)\leftrightarrow P(x_1,\dots,x_m)\big{)}$.

\begin{example}\label{Ex1} Assume that $L_0$ is the language of arithmetic with its standard interpretation. Let $R(x,y)$ be  formula $x=2y$, and let $R$ be the corresponding predicate  with domain $D_R=\mathbb N_0\times \mathbb N_0$.
Then the truth theories DTT and STT of the extension $\mathcal L_T$ of $L_0$ prove the universal $T$-schema\\
(UTR) \ $\forall x\forall y\big{(}T(\left\lceil R((\dot x,\dot y)\right\rceil)\leftrightarrow R(x,y)\big{)}$,\\
and the logical equivalences\\
(q1) \quad $q_1 x q_2 y T(\left\lceil R(\dot x,\dot y)\right\rceil)\leftrightarrow T(\left\lceil q_1 x q_2 y R(x,y)\right\rceil)\leftrightarrow
q_1 x q_2 yR(x,y)$,\\
(q2) \quad $q_1 x q_2 y\neg T(\left\lceil R(\dot x,\dot y)\right\rceil)\leftrightarrow\neg T(\left\lceil q_1 x q_2 yR(x,y)\right\rceil)\leftrightarrow q_1 x q_2 y\neg R(x,y)$.

The sentences in (q1) are true iff $q_1 q_2$ is $\forall  \exists$ or $\exists \exists$, and false iff  $q_1 q_2$ is $\forall  \forall$ or $\exists \forall$.
In (q2) the sentences are true iff $q_1 q_2$ is $\forall  \forall$ or $\forall \exists$, and false iff $q_1 q_2$ is $\exists \forall$ or $\exists \exists$.
\end{example}

The next example is a modification of an example presented in \cite[p. 704]{Ray18}.

\begin{example}\label{Ex2}  Let $L_0$, $\mathcal L_T$, $T$ and $R(x,y)$ be as in Example 6.1. Denote by $T'$ the predicate with domain $\mathbb N_0$ corresponding to formula $\exists yR(x,y)$. Let the sublanguage $L'_0$  of $L_0$ be formed by the syntax of $L_0$, predicate $T'$ and sentences $\exists y R({\bf n},y)$, where $\bf n$ goes through all numerals, the names of natural numbers $n$. Choose $n$ to be the code number of the sentence $\exists y R({\bf n},y)$ in $L'_0$ for each numeral $\bf n$. If $A$ is a sentence of $L'_0$, i.e. $\exists y R({\bf n},y)$ for some numeral $\bf n$, then $n$ is the code number of $A$, so that $\left\lceil A\right\rceil$ is $\bf n$.
Since $T'({\bf n})$ denotes the sentence $\exists y R({\bf n},y)$ for each numeral $\bf n$, then $T'(\left\lceil A\right\rceil)$ or equivalently $T'({\bf n})$, is true (respectively false) in $L'_0$ iff  $\exists y R({\bf n},y)$, or equivalently $A$,  is true, (respectively false) in $L'_0$ iff (by definition of $R(x,y)$) $n$ is even (respectively odd). Thus $T'$ is a truth predicate of $L'_0$. The restriction of $T$ to  $L'_0$ is also a truth predicate of $L'_0$. It is not equal to $T'$ because the coding of $L'_0$ is not the restriction of the G\"odel numbering of  $\mathcal L_T$, whereas $L_0$ can have  that syntax.
\end{example}

\section{Remarks}\label{R3}

Theories DTT and STT conform to the seven desiderata  presented in \cite{Lei07} for theories of truth (cf. \cite[Theorem 4.2]{Hei18}).
This challenges the current view (cf, e.g., \cite{Fe, Lei07, Raa}). 
Theories DTT and STT are also free from paradoxes.
The fact that the languages $\mathcal L_T$ which have these theories of truth contain their own truth predicates seems to be in contrast to
some Tarski's limiting theorems for theories of truth described, e.g., in \cite{Ray18}. Conformity of theories DTT and STT to the principles of classical logic can be vitiated by adding first-order syntax to $\mathcal L_T$. The object language $L_0$ is allowed to have that syntax.
\smallskip

Theorems \ref{T1} and \ref{T0} imply that theories DTT and STT of truth together contain the theory DSTT of truth introduced in \cite[Theorem 4.1]{Hei18} for languages $\mathcal L^0$ when  those sentences  of the form $A\vee B$, $A\rightarrow B$ and $\neg(A\wedge B)$ whose one component has not a true value are deleted from $\mathcal L^0$. Otherwise the languages $\mathcal L_T$ for which theories DTT and STT of truth are introduced extend languages $\mathcal L^0$. The amount of predicates and compositional sentences are multiplied by means of the added  non-truth predicate $\neg T$, and  predicates of the object language $L_0$ which have several free variables. 
\smallskip

The family of those languages which conform to classical logic is considerably larger than the families of those object languages considered in \cite{Hei18}. For instance, the object language $L_0$ can be any language whose  every sentence is valuated by its meaning either as true or as false. Every language $L_0$ which has properties (i) -- (iii), e.g., every countable and bivalent first-order language with or without identity, conforms to classical logic.  If $L_0$ is any countable language conforming to classical logic, and if $L'_0$ is any sublanguage of $L_0$ formed by the syntax of $L_0$, any nonempty subset of its sentences and any subset of its predicates, then $L'_0$ conforms to classical logic. For instance, the language $L'_0$ in Example \ref{Ex2} conforms to classical logic.
\smallskip

The set $U$ used in the definition of $\mathcal L_T$ is the smallest consistent set for which $U=G(U)$, where $G(U)$ is the set of G\"odel numbers of  sentences of $L(U)$. Thus $U$ is the minimal fixed point of the mapping $G:\mathcal C\to\mathcal C$, where $\mathcal C$ is the set of consistent sets of G\"odel numbers of  sentences of $\mathcal L$. Moreover $\mathcal L_T$ conforms to classical logic. Thus
the sentences of $\mathcal L_T$ are grounded in the sense defined by Kripke in \cite[p. 18]{15}. The language $\mathcal L_\sigma$ determined by the minimal fixed point in Kripke's construction contains also sentences which don't have truth values. For instance,  the sentence
$A\leftrightarrow T(\left\lceil A\right\rceil)$ has not a truth value for every sentence $A$ of $\mathcal L_\sigma$. Thus a three-valued logic is needed in \cite{15}, as well as in \cite{Fe} and in \cite{HH}. The only logic used here is classical.
\smallskip

The equivalence of truth values of sentences in theories DTT and STT show that the notion of 'grounded truth' defined by valuation (I) conforms to the 'ordinary' notion of truth.
\smallskip

In the metalanguage used in the above presentation some concepts dealing with predicates and their domains are revised from those used in \cite{Hei18} so that they agree better with the corresponding concepts in informal languages of  first-order logic (cf. \cite{1}).  The circular reasoning used in \cite{Hei18} to show that $G(U)$ is consistent if $U$ is consistent is corrected in the proof of Lemma \ref{L201}.
\smallskip

Mathematics, especially ZF set theory, plays a crucial role in this paper. Metaphysical necessity of pure  mathematical truths is considered in \cite{Lei18}.

\label{pagefin}
\end{document}